\documentclass[11pt,draftcls,onecolumn,doublespace]{IEEEtran}  

\ifCLASSINFOpdf
\else
\fi

\newcommand{\bx}{\mathbf{x}}

\usepackage{amsfonts}
\usepackage{mathrsfs}

\usepackage{amsmath,amssymb,url,amsfonts,mathrsfs}
\usepackage{enumerate}
\usepackage{verbatim}
\usepackage{pgf}

\usepackage{graphicx}
\usepackage[bookmarks=true,pageanchor,colorlinks,linkcolor=blue,anchorcolor=blue,citecolor=blue,urlcolor=blue,hyperfootnotes=false]{hyperref}
\usepackage[all]{hypcap}
\usepackage{setspace}

\long\def\blue#1{{\color{black}#1}}
\newtheorem{theorem}{{\bf Theorem}}[section]
\newtheorem{lemma}{{\bf Lemma}}[section]

\newtheorem{definition}{{\bf Definition}}[section]
\newtheorem{example}{{\bf Example}}[section]
\newtheorem{assumption}{{\bf Assumption}}[section]
\newtheorem{remark}{{\bf Remark}}[section]

\renewcommand{\v}[1]{\ensuremath{\boldsymbol{\mathrm{#1}}}}
\newcommand{\E}{\ensuremath{\mathbb{E}}}

\newcommand{\Z}{\mathbb{Z}}

\long\def\old#1{}

\begin{document}

\title{\blue{Dynamic} Scheduling for Charging Electric Vehicles: \blue{A Priority Rule}}
%
%

%


%



\author{Yunjian~Xu,
        Feng~Pan, Lang~Tong
{\thanks{Yunjian Xu is with the
Engineering Systems and Design Pillar,
        Singapore University of Technology and Design, Singapore (e-mail: yunjian\_xu@sutd.edu.sg).
Feng Pan is with the Electricity Infrastructure Group, Pacific Northwest National Laboratory, Richland, WA 99354, USA (e-mail: feng.pan@pnnl.gov). Lang Tong is with the School of Electrical
and Computer Engineering, Cornell University, Ithaca, NY 14853, USA (e-mail: ltong@cornell.edu).}
\thanks{This
research was supported in part by the MIT-SUTD International Design Center (IDC) Grant IDG21400103,
 and the National Science Foundation under CNS 1135844. A preliminary version of this technical note appeared
in the Proceedings of the 2012 IEEE Conference on Decision and Control \cite{XP12}.}}
}

\maketitle

\begin{abstract}
We consider the scheduling of multiple tasks with pre-determined deadlines
under random processing cost. This problem is motivated by
the potential of large scale adoption of plug-in (hybrid) electric vehicles {(PHEVs)} in the near future.
The charging requests of
 PHEVs usually have deadline constraints, and the electricity cost associated
with PHEV charging is usually random due to the
 uncertainty in both system load and renewable generation.
 We seek to properly schedule the battery charging of
multiple PHEVs so as to
 minimize the overall cost, which is
derived from the total charging cost and the penalty for not completing
charging before requested deadlines. Through a dynamic programming formulation, we
establish the Less Laxity and Longer remaining Processing time (LLLP)
principle that improves any charging policy on a sample-path basis,
when the non-completion
penalty is a convex function of the additional time needed to
fulfill the uncompleted request.
Specifically, the LLLP principle states that priority should be given to vehicles that have less
laxity and longer remaining processing times.
 Numerical results demonstrate that heuristic policies that violate the LLLP principle, for example,
the earliest deadline first (EDF) policy, can result in significant performance loss.
\end{abstract}


%
\IEEEpeerreviewmaketitle

\section{Introduction}

{We study the scheduling of multiple processors to perform
 tasks with deadlines under random processing cost.}
Each task requires a certain amount of  processing time before its deadline, and failure to
fulfill its request incurs non-completion penalty. Different from most existing
deadline scheduling models,  our formulation allows the instantaneous processing cost to be
 time-variant and stochastic. We seek to characterize an optimal
scheduling policy that minimizes the long-term expected total cost (the sum of
 processing cost and non-completion penalty).
 Although the results derived in this technical note generally apply to the aforementioned framework,
  we will
  focus on the scheduling of PHEV charging that may have significant impacts on both the reliability
and efficiency of the next generation electric power grids.

Becoming popular
 in many countries, PHEVs (plug-in hybrid electric vehicles) may achieve significant market share over the next decade.  However, the charging of a large number of
PHEVs can add considerable stress to an existing power grid,
  especially at the distribution network level \cite{CH10,LSA11}.
 The scheduling of
charging PHEVs receives much attention in recent years \cite{CF10,WA12,GC14, HJ15}.
To minimize the
load variance through PHEV charging, a few recent papers propose several
approaches based on game theoretic analysis \cite{MCH10,KH13} and decentralized
optimization \cite{GTL11}.
  Although dynamic programming based
approaches have been employed to study the optimal control of power
management for a single PHEV \cite{RI08,MF11},
there lacks a dynamic
framework (on the scheduling of charging multiple PHEVs) that explicitly incorporates the stochasticity in both PHEV arrivals and charging costs.

This work is intimately related to the literature on deadline scheduling. For a single processor scheduling problem, it is well
known that the earliest
deadline first (EDF) policy \cite{LL73} and the least-laxity first
(LLF) policy \cite{D74} are optimal, if it is feasible to finish all tasks
before their deadlines. When the completion of all tasks is not feasible,
it has been demonstrated that EDF and LLF may perform poorly \cite{L86}.\footnote{There is also a substantial literature on
deadline scheduling of multiple processors \cite{D89}; for a survey, see \cite{DB11}.}
Closer to the present work, the authors of \cite{BE89,BT97} conduct a dynamic programming based
approach to characterize
 optimal scheduling policies for
the delivery of messages that would extinct after their individual deadlines.
{In the aforementioned literature} processing capacity (of each individual processer) is usually assumed to be constant over the entire
operation interval.
 As noted in \cite{SG11}, the scheduling
 of PHEV charging is fundamentally different, since the
cost associated with PHEV charging is time-varying and stochastic (due to the
inherent volatility in renewable generation and system load).

In this note, we consider a system with multiple (possibly a large
number of) PHEVs and an underlying power grid with renewable
generation. A system operator schedules the charging of PHEVs so as
to minimize the long-run average cost. The formulated dynamic program
(DP) incorporates {\it arbitrary} randomness in
 both the charging cost and the PHEV arrival processes.

The main contribution of this technical note is to establish
\blue{an important and somewhat counter-intuitive (partial) characterization
on  optimal scheduling policies.}
In particular, we show the
{\bf Less Laxity and Longer remaining Processing time} (LLLP)
principle: priority should be given to vehicles that have less
laxity and
longer remaining processing times, if the non-completion penalty (as
a function of the additional time needed to complete the task) is
convex.\footnote{According to the LLLP principle, for two vehicles with
the same laxity, priority should be given to the vehicle with a
later deadline (and longer remaining processing time). This is
in sharp contrast to the case of a single processor with fixed
processing capacity, where the earliest deadline first (EDF) policy
is shown to be optimal.} For a given heuristic policy, we show that
an LLLP-based ``interchanging'' policy cannot be worse than the original heuristic.
 This result holds on every sample path and is robust against arbitrary random arrival process and
charging cost. We also show
(under some additional mild assumptions) the existence of an optimal
stationary policy that always gives priority to vehicles with less
laxity and longer remaining processing times. \blue{Numerical results presented in
Section \ref{sec:app} show that the LLLP principle is practically useful:} heuristic policies that violate the LLLP principle, such as
the well known earliest deadline first (EDF) policy, can result in significant performance loss.

\section{Model}\label{sec:model}

We consider an infinite-horizon {discrete time} model.
As in \cite{SG11}, we assume that each vehicle reports its arrival
time, departure time, and charging request to the  system operator at its
arrival. The system operator uses all information available at the
current stage
{({\it i.e.}, the current system state of the DP to be formulated in Section \ref{sec:DP} that includes the states of all
arrived vehicles, the operating condition of the power system, as well as the prediction on future PHEV arrivals)}
to schedule the charging of PHEVs.

 We study the scheduling problem of $N$ PHEV chargers.
  For $i=1,\ldots,N$, we refer to the
vehicle that is connected to the $i$th charger as vehicle $i$. At stage $t$,
let $\mathcal {I}_t \, {\subseteq} \, \{1,\ldots,N\}$ denote the set of chargers that are connected to electric vehicles,
and $|\mathcal {I}_t|$ denote the number of vehicles connected to chargers.
For each vehicle $i \in \mathcal {I}_t $, let $\alpha_i$ and
$\beta_i$ be its arrival and departure time, respectively. Under the
assumption that both arrival and departure occur at the
beginning of each stage, vehicle $i$ can be charged from stage
$\alpha_i$ through stage $\beta_i-1$.
We assume that $1\le
\beta_i-\alpha_i \le B$, i.e., every vehicle stays at a facility for
at least one stage, and at most $B$ stages.

 For every $i \in \mathcal {I}_t$,
let $\gamma_{i,t}$ denote its \textbf{remaining processing
time} at stage $t$, i.e., the number of time units of charging needed
to meet vehicle $i$'s charging request {under a time-invariant constant charging rate}.
{We assume that the processing time of each vehicle
is no greater than $E$.}
At stage $t$, for every $i \in \mathcal {I}_t$, we use a
two-dimensional vector, $x_{i,t} \buildrel \Delta \over
=(\lambda_{i,t}, \gamma_{i,t})$, to denote the \textbf{state of
vehicle $i$}, where $\lambda_{i,t}\buildrel \Delta \over =\beta_i-t$ is the number of remaining stages of the vehicle at
charger $i$. {For notational convenience, for $i \notin \mathcal {I}_t$ we let
 $x_{i,t}=(0,0)$.}

For every $i \in  \mathcal {I}_t$, $a_{i,t}=1$ if vehicle $i$ is charged at stage $t$,
and $a_{i,t}=0$ otherwise.
A feasible \textbf{action} at stage $t$, $\v{a}_t=(a_{1,t},\ldots,a_{N,t})$, is an
$N$-dimensional vector with $a_{i,t}
\le \gamma_{i,t}$ for every $i$.  Let $A_t$ denote the total number of vehicles charged at stage $t$, {\it i.e.},
 $A_t=\sum\nolimits_{i \in \mathcal {I}_t}
a_{i,t}$.  For a vehicle $i \in \mathcal {I}_t$, if it remains connected to charger $i$ at stage $t+1$,
its state {\bf evolves according to} $x_{i,t+1}=x_{i,t}-(1,a_{i,t})$.
For an empty charger $i \notin \mathcal {I}_t$, if a vehicle arrives at stage $t+1$
then $i \in \mathcal {I}_{t+1}$ and the state of charger $i$ becomes the initial state
of this vehicle.

At stage $t$, let $s_t \in \mathcal {S}$ denote the \textbf{state of grid}, where the set $\mathcal {S}$ is assumed to be finite.
The state of grid {incorporates all the currently available information on
all exogenous factors} that
have impacts on the cost associated with PHEV charging, such as the
level of renewable generation {and its prediction}, the system load {(excluding the PHEV charging load $A_t$) and its prediction}, and the current time.
The evolution of the state of grid depends on the current state $s_t$ and
the aggregate action\footnote{Note that charging a large number of vehicles may
influence the Independent System Operator's (ISO) dispatch and
reserve policy. To incorporate this type of impact, we allow the
evolution of the grid state to depend on the aggregate action in
general.} ${A}_t$. The \textbf{charging cost} at stage $t$ $C({A}_t,s_t)$ depends on the aggregate action $A_t$ and the state of grid $s_t$.

 At stage $t$, let $d_t \in \mathcal {D}$ be the \textbf{state of
demand}, where $\mathcal {D}$ is assumed to be finite.
The state of
demand $d_t$ contains all the currently available information
on future PHEV arrivals, and completely determines the
joint probability distribution on the number of
arrival vehicles in the future and their initial states.
The state of
demand evolves as a time-homogeneous Markov chain, whose state
transition is assumed to be independent of the state of the grid $s_t$ and the
action $\v{a}_t$.

\section{Dynamic Programming Formulation}\label{sec:DP}

In this section,  we formulate the
scheduling problem as an infinite-horizon dynamic program
(DP) by introducing its state space, admissible action set, transition
probabilities, stage cost, and average-cost objective function.

At each stage $t$, the \textbf{system state}, $\v{x}_t$, consists of
the states of all chargers, $\{x_{i,t}\}_{i=1}^N$, the state of grid
$s_t$, and the state of demand $d_t$. Let $\mathcal {X}$ denote the set of all possible
system states. Note that the size of state space grows exponentially with the number of
chargers, $N$.
Reasonable values of $N$, $B$, and $E$ lead to very high dimensions,
and make a direct solution to the DP impossible. We use $U_t(\v{x}_t)$ to denote the set of feasible
actions at stage $t$ under system state $\bx_t$.

The \textbf{transition probability} of the system state depends on the current system state,
$\v{x}$, and the current action $\v{a}_t$.\footnote{{Note that
while the evolution of vehicles' states certainly depends on
the action vector $\v{a}_t$, the evolution of $s_t$ depends only on the aggregate action $A_t$,
and the evolution of $d_t$ is completely exogenous.}} Since the state
transition is independent of the stage index $t$, we use
{$p_{\v{x},\v{y}} (\v{a}_t)$ to denote the transition probability from
state $\v{x}$ to $\v{y}$,} under the action $\v{a}$.

 At each stage $t$, the \textbf{stage cost}
$g(\v{x}_t,\v{a}_t)$ consists of two parts: the
charging cost $C({A}_t,s_t)$ and the non-completion penalty.
Let $\mathcal {J} (\v{x}_t)$
denote the set of vehicles that will leave at stage $t+1$, i.e.,
$\mathcal {J} (\v{x}_t) =\left\{ j  \in \mathcal{I}_t:
\lambda_{j,t}=1 \right\}.$
The stage cost
 function at stage $t$ is\footnote{The state $s_t$ can incorporate the maximum capacity constraint on all the $N$ PHEV chargers
by including an element $c_t$ such that the charging cost becomes higher
than the highest possible non-completion penalty $Nq(E)$ if $A_t > c_t$. On the other hand,
our formulation omits the power flow constraints within a distribution network and cannot incorporate capacity constraints on any subset of the $N$ chargers.}
\begin{equation}\label{eq:cost}
g(\v{x}_t,\v{a}_t)=
C({A}_t,s_t) + \sum\nolimits_{j \in \mathcal {J}(\v{x}_t) }
q(\gamma_{j,t}-a_{j,t}),
\end{equation}
where the penalty function $q: \Z_+ \to [0,\infty)$ with $q(0)=0$ maps
the number of uncharged battery units to its non-completion penalty (resulting from greenhouse gas emission
or/and customers' inconvenience). Since both the set of
system states and the set of feasible actions are finite, the stage
cost is bounded.

{A feasible {\bf policy} $\pi=\{\nu_0,\nu_1,\ldots,\}$ is a sequence of
decision rules such that $\nu_t(\v{x}_t) \in
U_t(\v{x}_t)$ for every $t$ and $\bx_t$.}
Given an initial system state $\v{x}_0$, the time-averaged cost
 achieved by a policy
$\pi$ is given by
\begin{equation}
J_{\pi}(\v{x}_0) \buildrel \Delta \over
 =  \mathop {\limsup }\limits_{T \to \infty }
 \frac{1}{T} \E_\pi\left\{ {\sum\nolimits_{t = 0}^{T - 1} {g(\v{x}_t,
\nu_t(\v{x}_t))} } \right\}, \label{eq:costJ}
\end{equation}
where the expectation is over the distribution of future system
state $\{\v{x}_t\}_{t=1}^{T-1}$ (induced by the policy $\pi$).
{Since the state evolution of the formulated DP does not depend on the
time index and} the state space is finite, there
exists an optimal stationary policy $\pi^*=\{\mu^*,\mu^*,\ldots\}$,
and the limit on the right hand side of \eqref{eq:costJ} exists \cite{B11}.

 \old{We write the minimum average cost
starting from state $\v{x}$ as
\begin{equation}\label{eq:costJS}
J_{\mu^*}(\v{x})=\mathop {\limsup }\limits_{T \to \infty }
\frac{1}{T} \E_\pi\left\{ {\sum\nolimits_{t = 0}^{T - 1} {g(\v{x}_t,
\mu^*(\v{x}_t))} }  \, \Big | \, \v{x}_0=\bx  \right\}.
\end{equation}}

{Next we give an illustrative example of the general DP framework constructed above.}

\begin{example}
Our formulation incorporates the
 objective of minimizing load variance (that has been extensively explored in the literature \cite{MCH10,GTL11}).
{In this special case, the state of grid $s_t$ is set to be the net system load ({\it i.e.}, the difference between system load and renewable generation) excluding PHEV charging.
The charging cost is given by:}
 $$
 C(A_t,s_t) = H(A_t+s_t),
 $$
 where $H(\cdot)$ is a {strictly} convex function that maps the total (net) system load to generation cost;
 a commonly used cost function is quadratic, e.g., $H(x)=x^2$ \cite{MCH10,GTL11}.
 Note that if the {incremental non-completion penalty $q(n)-q(n-1)$ is
 set to be larger than the incremental charging cost $C(A_t,s_t) - C(A_t-1,s_t)$, for every $n \ge 1$, $A_t \ge 1$, and $s_t \in \mathcal{S}$,
 the deadline requirement of each vehicle becomes a ``hard constraint'', in that it is optimal to fulfill all charging requests before their deadlines, as long as it is feasible
 to do so.} $\hfill \blacksquare$
\end{example}

\begin{remark}
{Although the state of grid $s_t$ and the state of demand $d_t$
are modeled as stationary Markov chains,
it is worth noting that the time dependency of the grid status (e.g., renewable generation and system load) and PHEV arrivals can be incorporated by
 including in the states $s_t$ and $d_t$ a periodic
Markov chain that describes the evolution of local time.}

{The DP framework constructed in Sections \ref{sec:model}
and \ref{sec:DP} is general. The only conditions
on $s_t$ required by the LLLP principle (that will be formally stated and proved in Theorem \ref{prop:update})
 are: i) the charging cost at each stage $t$ is of the form $C(A_t,s_t)$,
which depends only on the aggregate action $A_t$ and $s_t$,
and ii) the evolution of $s_t$ depends only on $A_t$ (but not on $a_{i,t}$ for any $i$). In other words,
the LLLP principle holds regardless of the detailed model used by the operator to describe the power grid dynamics
(e.g., information included in the state $s_t$, its evolution,
and the exact form of charging cost).}  $\hfill \blacksquare$
\end{remark}

\section{The LLLP Principle}\label{sec:op}

In this section we establish the main result of this technical note. In
Section \ref{sec:op1}, we first define a partial order over the set
of vehicle states: a vehicle with less laxity and
a longer remaining processing time has a higher-order state. For any
given (possibly non-stationary) heuristic policy that violates the
LLLP principle, we construct an interchanging policy that gives
priority to the vehicle with a higher-order state. We show that
on every sample path, the interchanging policy can only
reduce the ex-post (realized) cost, compared with the original heuristic (cf.
Theorem \ref{prop:update}). In Section \ref{sec:op3}, under some mild
assumptions on the evolutions of the grid state $s_t$ and the state
of demand $d_t$,
 we show the existence
of an optimal stationary policy that follows the LLLP principle.

\subsection{LLLP-based Interchanging Policy}\label{sec:op1}

 For every vehicle $i \in
\mathcal{I}_t$, its {\bf laxity} (at stage $t$) is defined by
\begin{equation}\label{eq:th}
\theta_{i,t}=  \left \{ \begin{array}{ll}
\lambda_{i,t}-\gamma_{i,t}, & {\rm if}\;\;
\gamma_{i,t} >0, \\
B, & {\rm if}\;\; \gamma_{i,t} = 0.
\end{array} \right .
\end{equation}
Note that for a vehicle $i$ with $\gamma_{i,t}>0$, its laxity
$\theta_{i,t} \in \{ 1-E,2-E,\ldots,B-1\}$ is the maximum
number of stages it can tolerate before the time it has to
be put on uninterrupted battery charging.
We are now ready to define a partial order over the set of all
possible vehicle states.

\begin{definition}\label{Def:part}
For two vehicles $i, \, j \in \mathcal{I}_t$, we say $i \preccurlyeq j$ (vehicle
$j$ has priority over $i$) if $j$ has less laxity and longer remaining processing time, i.e.,
$\theta_{i,t} \ge \theta_{j,t}$,
$\gamma_{i,t} \le \gamma_{j,t}$, and at least one of these two inequalities strictly holds. $\hfill \blacksquare$
\end{definition}

It is not hard to check that the relation $\preccurlyeq$ is reflexive,
antisymmetric, and transitive, and therefore is a partial order. We
also note that if vehicle $j$ has less laxity and a later
deadline than vehicle $i$, i.e., if $\theta_{i,t} \ge
\theta_{j,t}$ and $\lambda_{i,t} \le \lambda_{j,t}$, then we must
have $i \preccurlyeq j$.

At a system state $\bx_t$,  two vehicles $i$ and $j$ are
 {\it incomparable}, if
 $\theta_{i,t} \ge
\theta_{j,t}$ and $\gamma_{i,t} > \gamma_{j,t}$, or $\theta_{i,t} >
\theta_{j,t}$ and $\gamma_{i,t} \ge \gamma_{j,t}$. In this case,
which vehicle should have higher priority depends
on future system dynamics.
On the other hand, if
vehicle $j$ has priority over vehicle $i$, we argue that priority should always be given to
vehicle $j$, regardless of future system dynamics. This result requires the penalty function to be convex, as stated in the following assumption.
\begin{assumption}\label{A:convex}
The {incremental non-completion penalty is non-negative and non-decreasing},  i.e.,
$$
0 \le q(n)-q(n-1) \le q(n+1) - q(n),\;\;\; n=1,2,\ldots.
$$
\end{assumption}

The non-completion penalty may come from the inconvenience caused to customers
 as well as the potential
environmental damage caused by the emission of PHEVs' combustion
engines. We note that environmental damage is usually considered to be convex
with respect to greenhouse gas emission \cite{B12}.

\begin{definition}[An LLLP-based Interchanging Policy] \label{def:update}
Suppose that at some system state $\bx_t$, vehicle $j$ has priority over
$i$, and that a policy $\pi=\{\nu_0,\nu_1,\ldots\}$ charges vehicle
$i$ but not $j$. Let $W \buildrel \Delta \over =
\max\{\lambda_{i,t}, \lambda_{j,t}\}-1$.
 We now formally define the \textbf{interchanging policy} $\bar
\pi=\{\nu_0,\ldots,\nu_{t-1},\bar \nu_t,\bar
\nu_{t+1},\ldots\}$ (generated from the policy $\pi$ with respect to
vehicles $i$ and $j$ at state $\bx_t$) as follows.
\begin{enumerate}
\item  We first let $\bar \nu_k = \nu_k$ for $k\ge t$,
 and then update the sequence of decision rules $\{\bar \nu_k \}_{k=t}^{t+W}$ as follows.

 \item Policy
$\bar \pi$ charges $j$ instead of $i$ at state $\bx_t$. That
is, $\bar \nu_t(\bx_t)$ is the same as $\nu_t(\bx_t)$ except
that its $i$th component is $0$ and its $j$th component is $1$.

 Following the state
$\bx_{t}$, for any (realized) sequence of system states that would occur
with positive probability under the policy $\pi$,
$\{\v{x}_k\}_{k=t+1}^{t+W}$, there exists a corresponding
sequence of system states following the state-action pair
$(\bx_{t},\bar \nu_t(\bx_t))$, $\{\hat
{\v{x}}_k\}_{k=t+1}^{t+W}$. {The corresponding state $\hat
{\v{x}}_k$ differs from ${\v{x}}_k$ only on the states of vehicles
$i$ and $j$: $\hat {\gamma}_{i,k} = \gamma_{i,k}+1$ for
$k=t+1,\ldots,\beta_i-1$, and $\hat {\gamma}_{j,k} = \gamma_{j,k}-1$
for $k=t+1,\ldots,\beta_j-1$.}

\item For every $\{\v{x}_k\}_{k=t+1}^{t+W}$, let
$\mathcal{G}(\{\v{x}_k\}_{k=t+1}^{t+W}) \subseteq
\{t+1,\ldots, \min\{\beta_i, \beta_j\}-1 \}$ be the set of stages that policy $\pi$
charges vehicle $j$ but not $i$, before vehicle
$i$'s departure.
 If the set
$\mathcal{G}(\{\v{x}_k\}_{k=t+1}^{t+W})$ is empty, let\footnote{Lemma \ref{L:update} shows that
 for $k=t+1,\ldots,t+W$, whenever the policy $\pi$ charges vehicle $j$
at state $\bx_k$, it is feasible to charge vehicle $j$ at the
corresponding state $\hat \bx_k$, i.e., $\hat \gamma_{j,k}\ge1$.}
$\bar \nu_{k}(\hat {\v{x}}_k)= \nu_{k}( {\v{x}}_k)$, for
$k=t+1,\ldots,t+W$, i.e., the interchanging policy $\bar \pi$ agrees with the original policy
$\pi$ after stage $t$.

If $\mathcal{G}(\{\v{x}_k\}_{k=t+1}^{t+W})$ is not empty,
let $w$ be its minimal element. At stages $k=t+1,\ldots,w-1$,
let
$\bar \nu_{k}(\hat {\v{x}}_k)= \nu_{k}( {\v{x}}_k)$. At stage
$w$, policy $\bar \pi$ charges vehicle $i$ instead of $j$,
i.e., $\bar \nu_w (\hat \bx_w)$ is the same as $\nu_w (\bx_w)$
except that its $i$th component is $1$ and its $j$th component is
$0$. $\hfill \blacksquare$
\end{enumerate}
\end{definition}

\begin{lemma}\label{L:update}
An interchanging policy $\bar \pi$ is feasible.
\end{lemma}

The proof of Lemma \ref{L:update} is given in Appendix A, where we show that the action taken by policy $\bar \pi$ at every system state is feasible.

\begin{theorem}[The LLLP Principle]\label{prop:update}
Suppose that Assumption \ref{A:convex} holds \blue{and
that, at some system state $\bx_t$, vehicle $j$ has priority over $i$ in the sense of Definition
\ref{Def:part}.}
Let $\bar \pi$ be
the interchanging policy generated from a policy $\pi$ \blue{according to Definition \ref{def:update}}. For every $T \ge \max\{\lambda_{j,t},\lambda_{i,t}\}-1$
and along every
sample path from stage $t+1$ through stage $t+T$,
 the total (realized) cost resulting from the interchanging policy
$\bar \pi$ cannot be higher than that achieved by the original policy $\pi$.
$\hfill \square$
\end{theorem}

\begin{proof}
 It follows
from Definition \ref{def:update} that the interchanging policy always
charges an equal number of vehicles as the original policy.
Since the evolution of
 $\{s_t\}$ depends only on the total
number of charged vehicles at each stage $t$, the two policies result in the same system dynamics.
As a result, following the state $\bx_t$, for every sequence of system states
that would occur with positive probability under the policy $\pi$,
$\{\v{x}_k\}_{k=t+1}^{t+W}$, there exists a corresponding sequence of
system states, $\{\bar {\v{x}}_k\}_{k=t+1}^{t+W}$, which will occur
with equal probability under policy $\bar \pi$.
If the set\footnote{The set
$\mathcal{G}(\{\v{x}_k\}_{k=t+1}^{t+W})$ and the sequence $\{\hat
{\v{x}}_k\}_{k=t+1}^{t+W}$ are introduced in Definition
\ref{def:update}.}
$\mathcal{G}(\{\v{x}_k\}_{k=t+1}^{t+W})$ is empty, then $\{\bar {\v{x}}_k\}_{k=t+1}^{t+W} =\{\hat
{\v{x}}_k\}_{k=t+1}^{t+W} $; otherwise, we have
 $$
 \{\bar {\v{x}}_k\}_{k=t+1}^{t+W} =\{\hat
{\v{x}}_{t+1},\ldots,\hat {\v{x}}_w, {\v{x}}_{w+1},\ldots, {\v{x}}_{t+W}
\},
$$
where $w$ be the minimum element in the set
$\mathcal{G}(\{\v{x}_k\}_{k=t+1}^{t+W})$.

We further note that the two policies, $\pi$ and $\bar \pi$, are identical
after stage $t+W$. As a result, to prove this theorem, it suffices to
show that for every realization of system states under the policy
$\pi$, $\{\v{x}_k\}_{k=t+1}^{t+W}$, and the corresponding realization of
system states under the policy $\bar \pi$, $\{\bar
{\v{x}}_k\}_{k=t+1}^{t+W}$,
\begin{equation}\label{eq:update2}
\begin{array}{l}
 \displaystyle g(\v{x}_t,\bar \nu_t(\v{x}_t)) +
{\sum\nolimits_{k = t+1}^{t+W}{
g(\bar {\v{x}}_k,  \bar \nu_k(\bar {\v{x}}_k))} }\le g(\v{x}_t, \nu_t(\v{x}_t)) +
{\sum\nolimits_{k = t+1}^{t+W} { g( {\v{x}}_k, \nu_k( {\v{x}}_k))}  }.
\end{array}
\end{equation}
We now prove Eq. \eqref{eq:update2} by discussing  the following two
cases:

\noindent {\rm 1.}
{If $\mathcal{G}(\{\v{x}_k\}_{k=t+1}^{t+W})$ is not empty,
 for every pair of  system state realizations,}
$\{{\v{x}}_k\}_{k=t+1}^{t+W}$ and $\{\bar {\v{x}}_k\}_{k=t+1}^{t+W}$,
 both polices must result in the same ex-post cost, i.e., the equality holds in
\eqref{eq:update2}.

\noindent {\rm 2.} If $\mathcal{G}(\{\v{x}_k\}_{k=t+1}^{t+W})$
is empty, then {whenever $\pi$ charges $j$, it must also charge
$i$, for $k=t+1,\ldots,\min\{\beta_i, \beta_j\}-1$.}
For a sequence of system states realized under $\pi$, $\{{\v{x}}_k\}_{k=t+1}^{t+W}$, let
$\rho_{i}$ denote the remaining processing time of vehicle
$i$ at its deadline $\beta_i$.
{That is, $\rho_{i} \buildrel \Delta \over =
\gamma_{i,\beta_i-1}-a_{i,\beta_i-1}$, where $a_{i,\beta_i-1}$ is
the action on vehicle $i$ at state $\bx_{\beta_i-1}$ according to
policy $\pi$. }Similarly, $\bar {\rho}_{i}$ is defined
for the corresponding sequence
$\{\bar {\v{x}}_k\}_{k=t+1}^{t+W}$ under the interchanging policy $\bar \pi$.
{Since $j$ has priority over $i$  at $\bx_t$,
it is straightforward to check (from the definition of $\bar \pi$ in Definition \ref{def:update})
that $0 \le \rho_{i} <\rho_{j}$, $\bar \rho_{j}=\rho_{j}-1$, and $\bar
\rho_{i}=\rho_{i}+1$.}
 It follows from Assumption
\ref{A:convex} that
$q(\bar \rho_{j}) +q(\bar \rho_{i}) \le q( \rho_{j}) +q(
\rho_{i})$.
Note that this inequality implies the desired result in \eqref{eq:update2},
since the two policies, $\pi$ and $\bar \pi$, result in the
same cost except possible different penalties for not fulfilling
vehicle $i$'s and $j$'s charging requests.
\end{proof}

{Before ending this section we make some brief discussion on the intuition behind the LLLP principle}.
We consider a simple {\bf two-vehicle example}. At stage $0$,
 the states of the two
vehicles are $x_{1,0}=(2,1)$ and $x_{2,0}=(3,2)$. Note that vehicle $2$ has priority to vehicle $1$, according to the LLLP principle.
If vehicle $1$ is charged at stage $0$, then the vehicle is fully charged (and not available for charging) at stage $1$;
on the other hand, if only vehicle $2$ is charged at stage $0$, then both vehicles are available for charging
at stage $1$.
The LLLP principle argues
that the latter situation is preferable, because
\begin{itemize}
\item in the latter situation, the operator has a larger set of
feasible actions at stage $1$;

\item under convex penalty functions,
remaining processing time should be split among multiple vehicles.
\end{itemize}

Under random charging cost and convex non-completion penalty, it is always desirable to have a larger number
 of smaller unfinished tasks that can be processed simultaneously when charging cost becomes lower
  in the future.

\begin{remark}
{Although an interchanging policy $\bar \pi$ cannot be worse than the original
heuristic, it may still be (sometimes obviously) suboptimal, since
it does not fully utilize the extra ``flexibility'' provided by the
LLLP principle. Note that $\bar \pi$ charges $i$ but not $j$
at stage $w$ (cf. Definition \ref{def:update}). A natural way to improve
 $\bar \pi$ is to charge both $i$ and $j$ at stage $w$ under certain circumstances, e.g.,
when the laxity of vehicle $j$ is small or it is cheap to charge an additional vehicle at stage $w$.}

{This intuition can be illustrated by the aforementioned two-vehicle example. The charging cost at stage $0$,
$1$, $2$ is $A_0$, $0$, and $2A_2$, respectively (here $A_t$ denotes the number of vehicles
charged at stage $t$). Consider an EDF policy $\pi$ that charges vehicle $1$ at stage $0$,
and charges vehicle $2$ at stages $1$ and $2$. According to Definition \ref{def:update},
 $t=0$ (when policy $\pi$ violates the LLLP principle) and $w=1$.
The unique optimal policy gives priority to vehicle $2$ at stage $t$ (following the LLLP principle), and charges both vehicles at stage $w$.}    $\hfill \blacksquare$
\end{remark}

\subsection{Optimality of the LLLP Principle}\label{sec:op3}

In this subsection, we show the existence of an optimal
 stationary policy that always follows the LLLP principle.
The following technical assumption is made to guarantee {that
the minimum average cost
does not depend on the initial state.}

\begin{assumption}\label{A:ini}
We assume the following.
\begin{itemize}
\item [{\rm \ref{A:ini}.1.}] Under every $d \in \mathcal{D}$, there is positive probability
that no vehicle arrives at the next stage.

\item [{\rm \ref{A:ini}.2.}] There exists a special state of grid $\bar s \in
\mathcal{S}$ such that for some positive integer $m$,  for every
initial state $s_0 \in \mathcal{S}$, and under the
sequence of zero aggregated charging decisions,
$\{A_0=0,A_1=0,\ldots,A_{m-1}=0\}$, state $\bar s$ is visited
with positive probability at least once within the first $m$ steps.

\item [{\rm \ref{A:ini}.3.}]
 The state of demand $\{d_t\}$ evolves as an ergodic Markov chain.
\end{itemize}
\end{assumption}

Note that Assumption \ref{A:ini}.2 holds if without PHEV charging,
(e.g., in a model with $N=0$), the state of grid evolves as an
ergodic Markov chain. Although Assumption \ref{A:ini}.3 requires that
$\{d_t\}$ is ergodic,
the  time dependency of PHEV arrival process can be incorporated by
 including in the state $d_t$ a periodic
Markov chain that describes time evolution.

\begin{lemma}\label{prop:ini}
Suppose that Assumption \ref{A:ini} holds. The minimum average cost
is equal for all initial states, i.e.,
\begin{equation}\label{eq:ini}
\lambda \buildrel \Delta \over = J_{\mu^*}(\v{x}),\;\;\;\;\forall\;
\v{x} \in \mathcal{X},
\end{equation}
where $J_{\mu^*}(\v{x})$ is \blue{the (minimum) average cost achieved
by an optimal stationary policy $\mu^*$ (cf. its definition in Eq. \eqref{eq:costJ}).}
\end{lemma}

\begin{proof}
We pick up a state $\bar d \in \mathcal{D}$, and define a
\textbf{special system state}
\begin{equation}\label{eq:spe}
\bar {\bx} = \{(0,0),\ldots,(0,0), \bar s, \bar d\},
\end{equation}
where $\bar s \in \mathcal{S}$ is the special state of grid
defined in Assumption \ref{A:ini}.2. At this special system state,
all charging facilities are empty with state $(0,0)$.
According to Proposition 7.4.1 of \cite{B11a}, to
show the desired result we only need to argue that, for every initial
system state $\bx_0 \in \mathcal{X}$ and under all policies, the
special system state is visited with positive probability at least
once within the first $L \buildrel \Delta \over =
B+\max\{m,|\mathcal{D}|\}$ steps, with $m$ being the integer defined
in Assumption \ref{A:ini}.2.\footnote{{This special state
is recurrent in the Markov chain induced by every stationary policy \cite{B11a},
and therefore the minimum average cost with any initial state cannot be different
from that with the special state being the initial state. Due to the space limit we omit the detailed proof and readers can refer to the proof of Prop. 7.4.1 in \cite{B11a}.}}

Under Assumption \ref{A:ini}.1, the probability that no vehicle arrives within the first $L+1$
stages is positive, and if this is the case, all charging facilities
are empty from stage $B$ to stage $L$, regardless of the policy used
by the operator. In this case, since no vehicle is charged from stage $B$ to stage $L-1$,
Assumption \ref{A:ini}.2 implies that with positive probability
the special state $\bar s$ is visited at least once from stage $B$ to
stage $L$. Also, since $\{d_t\}$
is an ergodic Markov chain, the probability that the state
$\bar d$ is visited at least once from stage $B$ to stage $L$
is positive. Since the evolutions of $\{d_t\}$ and $\{s_t\}$ are
assumed to be independent, for all initial system states and under all
policies, the special system state $\bar {\bx}$ is visited with positive
probability at least once from stage $B$ to stage $L$.
\end{proof}

\begin{theorem}\label{prop:property}
Suppose that Assumptions \ref{A:convex} and \ref{A:ini} hold. There
exists an optimal stationary policy $\mu^*$ that always follows the
LLLP principle. That is, at every system state $\v{x} \in
\mathcal{X}$, if the $i$th component of $\mu^*(\v{x})$ is $1$
(vehicle $i$ is charged), then for every vehicle $j$ such that $i
\preccurlyeq j$ at $\v{x}$, the $j$th component of $\mu^*(\v{x})$
must also be $1$ (vehicle $j$ must be charged).
\end{theorem}

The proof of Theorem \ref{prop:property} is given in Appendix B.
Since the state space is finite, there
exists an optimal stationary policy (cf.
page $175$ of \cite{B11}).
The crux of our proof centers
on showing that any optimal stationary policy can be mapped to
an optimal stationary policy that follows the LLLP principle, through an optimality condition (Bellman's equation)
based argument.

\old{Since the
state space is finite, by repeating this interchange argument for finitely many times, we can
construct an optimal stationary policy that follows the LLLP
principle at all system states.}

{
\section{Numerical Results}\label{sec:app}
In this section, we compare the performance of three stationary
heuristic policies, the EDF (Earliest Deadline First) policy and
two LLF (Least Laxity First)-based heuristic policies. We consider a
case with $400$ chargers, i.e., $N=400$ (large enough to accept all arriving vehicles in
our simulation).
The state of grid
reflects the maximum capacity available for PHEV charging, i.e.,
the cost function associated with $s \in \mathcal{S}$ is given
by
$$
C(A,s)=0, \quad  {\rm if}\; A \le s;  \qquad C(A,s)=Nq(E), \quad {\rm
if}\; A > s,
$$
{where $Nq(E)$ is an upper bound on the highest possible non-completion penalty that could incur
to all vehicles in the set $\mathcal{I}_t$.}
Obviously, the operator should never charge more than $s_t$ vehicles
at stage $t$.
The states of grid, $\{s_0,s_1,\ldots\}$,
are assumed to be independent and {identically distributed} random variables that are
uniformly distributed over $\mathcal{S}=\{40,41,\ldots,160\}$.
Since we have assumed zero charging cost as an approximation for
 the case where the charging cost is much smaller than the non-completion penalty,
in our simulation the only source of cost is non-completion penalty.

For simplicity,
we consider a case where the number of arriving vehicles is a time-invariant
constant, and the initial states of arriving vehicles are
independent and identically distributed random variables. In
particular, the number of stages for which a newly arrived vehicle
$i$ will stay at a charging facility, $\beta_i-\alpha_i$, is
uniformly distributed over the set $\{1,\ldots,10\}$ (i.e., $B=10$),
and the time needed to fulfill its request,
$\gamma_{i,\alpha_i}$, is
 uniformly distributed over the set
$\{1,\ldots,\beta_i-\alpha_i\}$.

For a system state $\bx_t$, let $V(\bx_t) $ be the number of vehicles in the set $\mathcal{I}_t$
that are not fully charged. At a system state $\bx_t$, the stationary EDF policy
 charges the first $\min\{s_t,V(\bx_t)\}$ vehicles with
the earliest departure times. For two vehicles that have the same
deadline, $\pi$ charges the one with less laxity.
At a system state $\bx_t$, both LLF-based policies charge the
first $\min\{s_t,V(\bx_t)\}$ vehicles with the least laxity.
For two vehicles with the same laxity, the LLSP (Least Laxity and
Shorter remaining Processing time) policy  gives
priority to the vehicle with shorter remaining processing time
(an earlier departure time), while the LLLP (Least Laxity and Longer
remaining Processing time) policy gives
priority to the vehicle that has longer remaining processing time.

\old{
\begin{figure}\label{Fig:phev} \vspace{-5mm}
\begin{minipage}[t]{0.499\textwidth}
\centering
\includegraphics[width=2.8in]{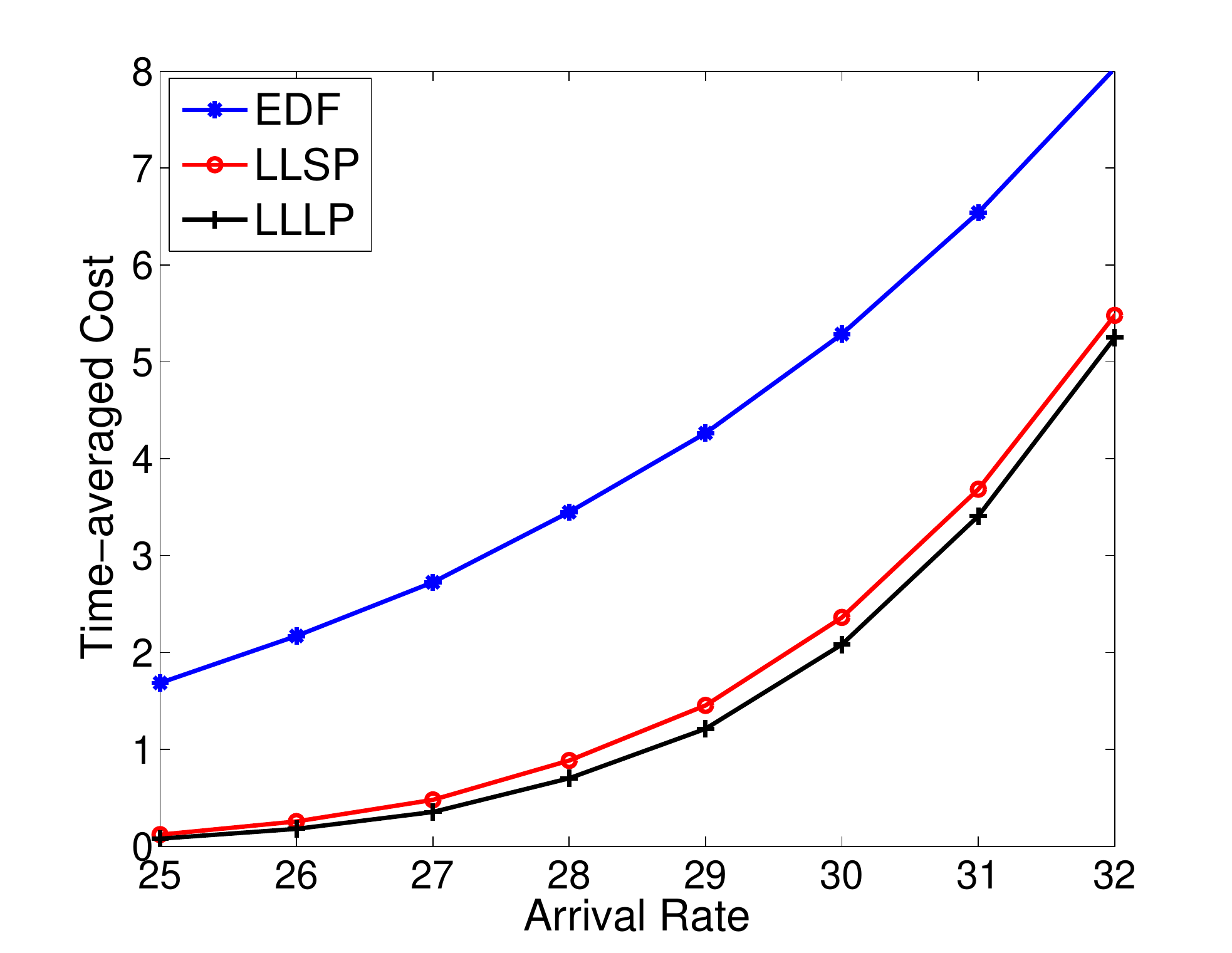}
\end{minipage}%
\begin{minipage}[t]{0.499\textwidth}
\centering
\includegraphics[width=2.8in]{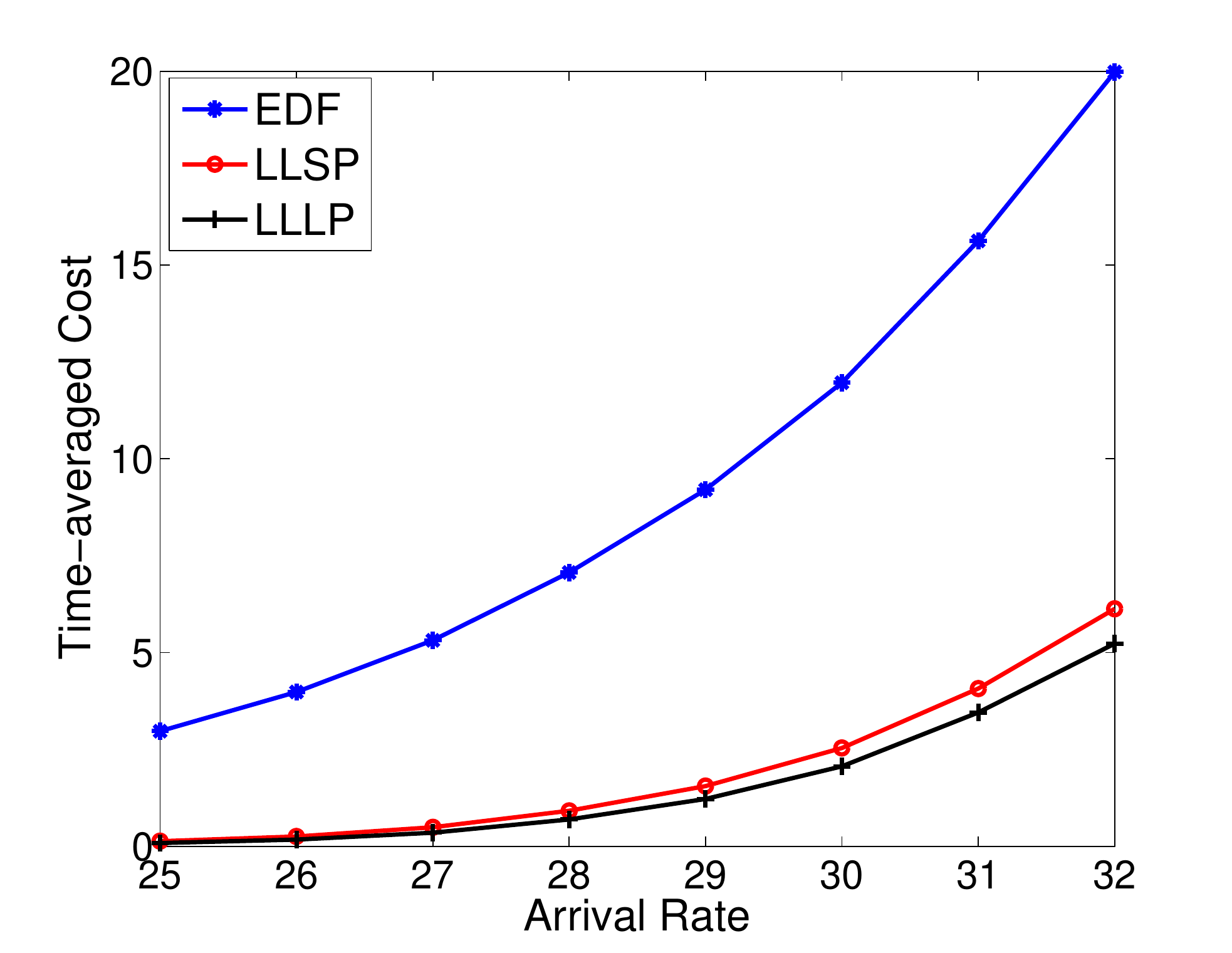}
\end{minipage}
 \vspace{-5mm}
    \caption{A simulation experiment with $1,500,000$ trajectories for each arrival rate
    on the horizontal axis and time-averaged cost on the vertical axis, {for two penalty functions
$q(n)=n$ (the left figure) and $q(n)=n^2$ (the right figure).}}  \vspace{-5mm}
\end{figure}}

\begin{figure}\label{Fig:phev}
\centering
\includegraphics[width=2.7in]{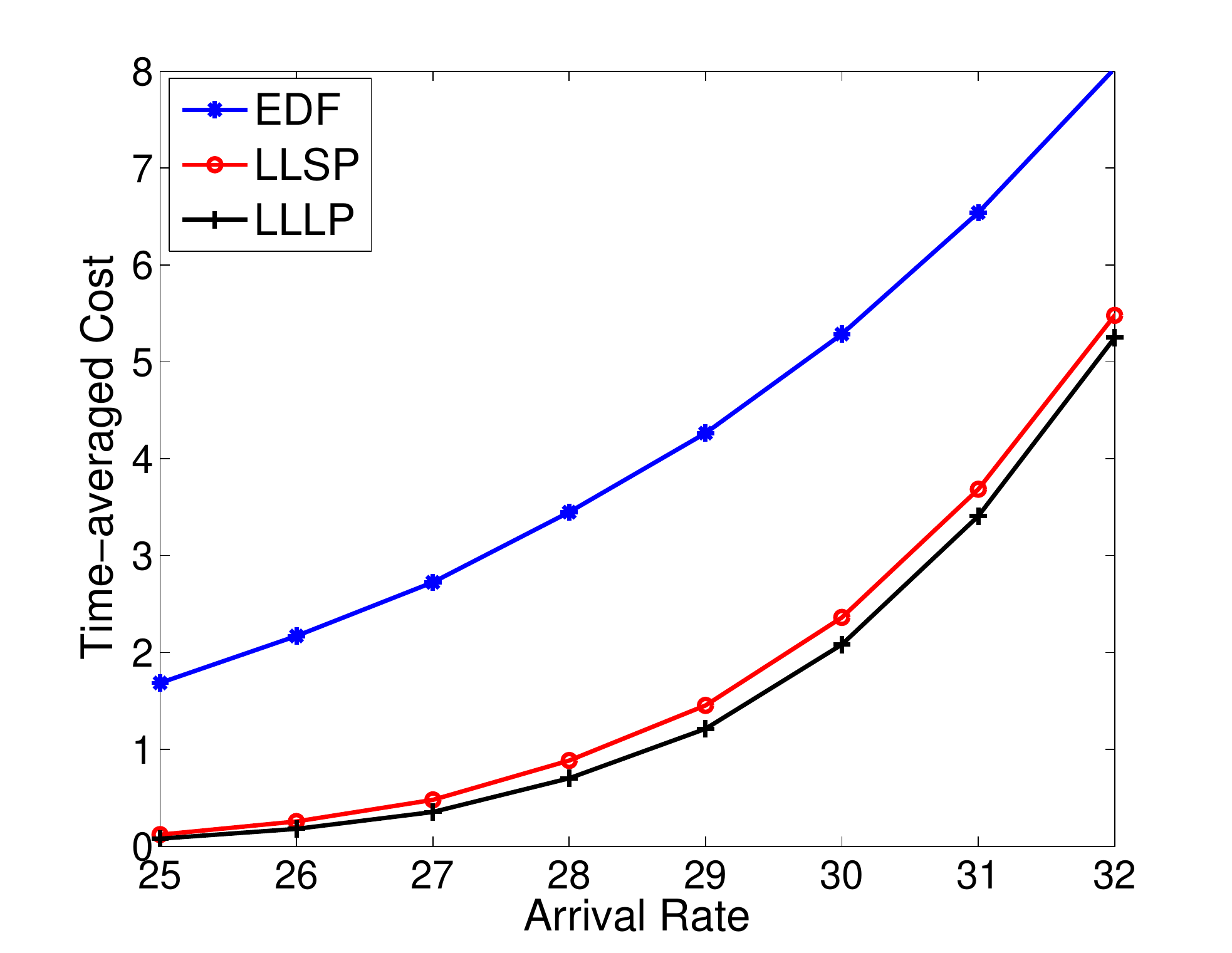}\vspace{-2mm}
   \caption{A simulation experiment with $1,500,000$ trajectories for each arrival rate
    on the horizontal axis and time-averaged cost on the vertical axis, with non-completion penalty
$q(n)=n$.}\vspace{-2mm}
\end{figure}

\begin{figure}\label{Fig:phev}
\centering
\includegraphics[width=2.7in]{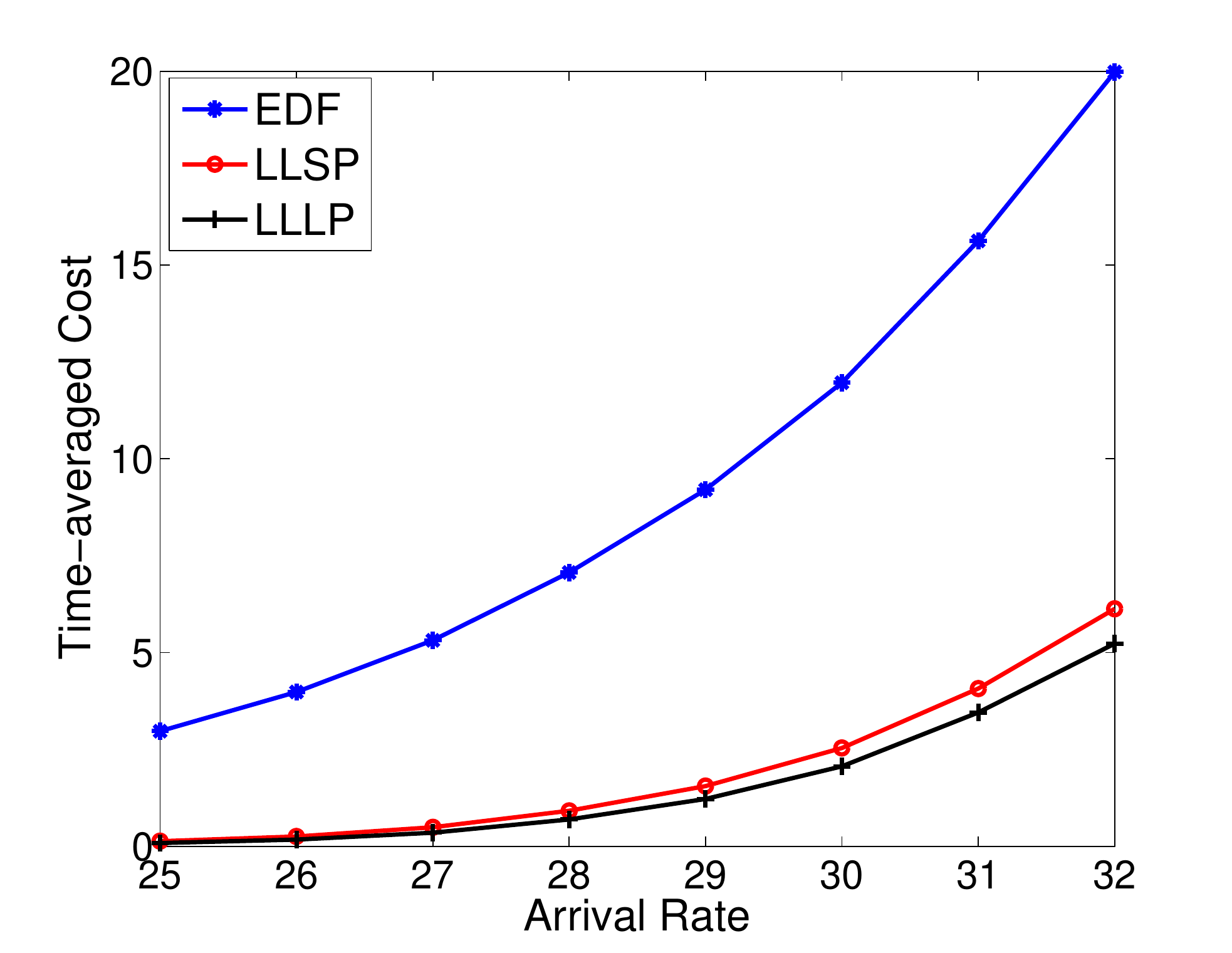}\vspace{-2mm}
   \caption{A simulation experiment with $1,500,000$ trajectories for each arrival rate
    on the horizontal axis and time-averaged cost on the vertical axis, with non-completion penalty
{$q(n)=n^2$}.}\vspace{-3mm}
\end{figure}

The time-averaged cost resulting from the
three heuristic policies (EDF, LLSP, and LLLP) are compared in Fig.\
1 and Fig. 2, {for two different non-completion penalty functions $q(n)=n$ and $q(n)=n^2$.}
 For both penalty functions, the numerical results show that the LLLP policy achieves the
lowest time-averaged cost, and that LLSP significantly outperforms EDF.
{We note that the performance gap between LLSP and LLLP is much more significant under quadratic
 non-completion penalty. This is because the LLLP policy distributes the total remaining processing time to a larger number of vehicles with smaller remaining processing times,
 which in turn leads to lower non-completion penalty when the penalty function $q()$ is strictly convex.
Indeed, under the linear penalty function, the LLLP policy reduces the
time-averaged cost by $15 \%-35\%$ (compared to the LLSP policy) when the arrival rate is less than $30$;
while under the quadratic penalty function, the performance gap between LLSP and LLLP is much larger, and remains
above $15 \%$ (of the cost resulting from LLSP) even when the arrival rate ranges in $\{30, 31, 32\}$.}
}

\section{Conclusion}\label{sec:con}

We formulate the scheduling problem of charging multiple PHEVs
 as a
Markov decision process. Using an interchange argument, we prove the less laxity and longer remaining processing
time (LLLP) principle: priority should be given to vehicles that
have less laxity and longer remaining processing times, if the
non-completion penalty function is convex and the operator does not
discount future cost. {We note that the LLLP principle is a partial characterization on the optimal scheduling policy,
and that there may exist many stationary policies that do not violate the LLLP principle.
A plausible future research direction is to compare and rank these heuristic
policies in stylized models with more structures in system dynamics.}

\old{For a given heuristic policy, we
 introduce two forms of interchanging
 polices generated according to the LLLP principle. We
 show that an interchanging policy must outperform
 the original heuristic.
Numerical results demonstrate that heuristic policies that violate
the LLLP principle (e.g., the EDF policy) may result in significant
losses of system efficiency.}

\appendices
\section{Proof of Lemma \ref{L:update}}\label{sec:A0}
To argue the feasibility of the interchanging policy $\bar \pi$, we will show that
\begin{enumerate}
\item in period $w$ when the original policy $\pi$ first charges vehicle $j$ but not $i$, it is feasible for the interchanging policy to charge vehicle $i$;

\item  in period $k=t+1,\ldots,w-1$, whenever the original policy $\pi$ charges vehicle $j$, it is feasible for the interchanging policy to charge vehicle  $j$.
\end{enumerate}

It is straightforward to check the first point, i.e., at the stage $w$ it is feasible for the interchanging policy $\bar \pi$ to charge vehicle $i$.
This is because the original policy $\pi$ charges vehicle $i$ at stage $t$ but the interchanging policy does not,
and the interchanging policy does not
charge vehicle $i$ whenever the original policy $\pi$ does not, at every stage before $w$.

We now prove the second point.
We first consider the case
where the original policy $\pi$ first charges vehicle $j$ but not $i$ at stage $w$.
Since at state $\v{x}_t$, vehicle $j$ has priority over $i$, and the policy $\pi$
charges vehicle
 $i$ but not $j$, we must have
 $\theta_{j,t+1}<\theta_{i,t+1}$ and $\gamma_{j,t+1}>\gamma_{i,t+1}$ at state $\bx_{t+1}$.
Before stage $w$ (for $k=t+1,\ldots,w-1$),
whenever the policy $\pi$ charges vehicle $j$ at state $\bx_k$, it
also charges vehicle $i$. It follows that for
$k=t+1,\ldots,w-1$, whenever the policy $\pi$ charges vehicle
$j$ at $\bx_k$, we must have $\gamma_{j,k}>\gamma_{i,k} \ge
1$, which implies that $\hat \gamma_{j,k} = \gamma_{j,k}-1 \ge 1$,
i.e., it is feasible for the interchanging policy to charge $j$.

A similar argument applies to the case
where the set $\mathcal{G}(\{\v{x}_k\}_{k=t+1}^{t+W})$ is empty ($w = \infty$),
and the deadline of vehicle $j$ is no later than  $i$'s. We have $W=\beta_i-1$.
Before vehicle $i$'s departure (for $k=t+1,\ldots,\beta_i-1$),
whenever the policy $\pi$ charges vehicle $j$ at state $\bx_k$, it
also charges vehicle $i$. It follows that for
$k=t+1,\ldots,\beta_i-1$, whenever the policy $\pi$ charges vehicle
$j$ at  $\bx_k$, we  have $\gamma_{j,k}>\gamma_{i,k} \ge
1$, and therefore $\hat \gamma_{j,k} = \gamma_{j,k}-1 \ge 1$.

We finally consider the case where $w=\infty$ and the deadline of $j$ is later than
$i$'s. We have $W=\beta_j-1$.
 At state $\bx_{\beta_i-1}$, we know vehicle $j$'s laxity must
 be strictly less than vehicle $i$'s, i.e.,
$$
\theta_{j,\beta_i-1}=\lambda_{j,\beta_i-1}-\gamma_{j,\beta_i-1}<\lambda_{i,\beta_i-1}-\gamma_{i,\beta_i-1}
=1-\gamma_{i,\beta_i-1}.
$$
If $\gamma_{i,\beta_i-1} \ge 1$, then we have
$\theta_{j,\beta_i-1}<0$. If $\gamma_{i,\beta_i-1}=0$, since the
policy $\pi$ does not
 charge vehicle $i$ at state $\bx_{\beta_i-1}$,
it does not charge vehicle $j$, because the set
$\mathcal{G}(\{\v{x}_k\}_{k=t+1}^{t+W})$ is empty. In this case,
 we have
$\theta_{j,\beta_i-1}\le 0$ and $\theta_{j,\beta_i}<0$. In either
case,  for $k=\beta_i,\ldots,\beta_j-1$, we have $\theta_{j,k}<0$.
It follows that for $k=\beta_i,\ldots,\beta_j-1$, $\hat \gamma_{j,k}
= \gamma_{j,k}-1 \ge 1$.

\section{Proof of Theorem \ref{prop:property}}\label{sec:C}

We first introduce a necessary and sufficient condition for the optimality of a stationary policy. This condition will be used later
in the proof.
 Under Assumption \ref{A:ini}, the minimum average cost $\lambda$,
 together with an $|\mathcal{X}|$-dimensional vector $\v{h}=\{h(\v{x})\}_{\bx\in
 \mathcal{X}}$, satisfies the following Bellman's
equation:
\begin{equation}\label{eq:Bell}
\lambda+h(\v{x})=\min_{\v{a}\in U(\v{x})} \left\{  g(\v{x},\v{a}) +
\sum_{\v{y} \in \mathcal{X}} p_{\v{x},\v{y}} (\v{a})
h(\v{y})\right\},\;\;\;\forall \v{x} \in \mathcal{X},
\end{equation}
where $U(\v{x})$ denotes the set of feasible actions at system state
$\bx$. It is known that a stationary policy $\mu^*$ is optimal,
i.e., $J_{\mu^*}(\bx) = \lambda$ for all $\bx$, if $\mu^*(\bx)$
attains the minimum in \eqref{eq:Bell} for each $\bx \in
\mathcal{X}$  \cite{B11}.

Let $(\lambda,\v{h})$ be a solution to the Bellman's
equation in \eqref{eq:Bell}. For every $\bx \in  \mathcal{X}$,
 $h(\bx)$ (usually referred to as the differential cost
for state $\bx$) is the minimum, over all policies, of the
difference between the expected cost to reach the special state
$\bar {\bx}$ (cf. \eqref{eq:spe}) from $\bx$ for the first time and the cost that
would be incurred if the cost per stage were equal to the minimum
average $\lambda$ at all states.  \blue{To formally define $h(\bx)$, consider a
modified dynamic program (DP) that is the same as the DP formulated in
Section \ref{sec:DP}, except that the special system state is
absorbing ($\tilde p_{\bar {\bx}, \bar {\bx}}(\v{a})=1$
for every action vector $\v{a}$), and the cost associated with the
special system state is $\lambda$ ($\tilde g (\bar {\bx},
\v{a})=\lambda$ for every feasible action $\v{a}$). Here, $\tilde p$ and $\tilde g$ denote the state
transition and stage cost functions of the modified Markov decision process, respectively.} The differential
cost for state $\bx$ can be written as
\begin{equation}\label{eq:h}
h(\v{x})\buildrel \Delta \over =\min_{\pi} \mathop {\limsup }\limits_{T \to \infty }
\E\left\{ {\sum\limits_{t = 0}^{T - 1}( {\tilde g(\v{x}_t,
\nu_t(\v{x}_t))} - \lambda) } \, \big | \, \v{x}_0=\v{x} \right\},
\end{equation}
where $\pi=\{\nu_0,\nu_1,\ldots,\}$. In the proof of Lemma
\ref{prop:ini} we have shown that for any initial state $\bx$, the
special system state is visited with positive probability at least
once within the first $(B+\max\{m,|\mathcal{D}|\})$ steps, which
implies that $h(\v{x})$ is finite for every $\bx \in \mathcal{X}$.
\old{If we take ${\tilde g(\v{x}_t, \nu_t(\v{x}_t))} - \lambda$ as
the stage cost function, the modified DP is actually a Stochastic
Shortest Path (SSP) problem with the termination state $\bar
\bx$.}

We next employ a simple interchange argument to prove the existence of
an optimal stationary policy that follows the LLLP principle.
Let $\mu$ be an optimal stationary policy, and suppose that there
exists a system state $\v{x}_t$ such that the $i$th component of
$\mu(\v{x}_t)$ is $1$, the $j$th component of $\mu(\v{x}_t)$ is $0$,
and that $i \preccurlyeq j$. Since $\gamma_{i,t} \ge 1$ and $i
\preccurlyeq j$, we must have $\gamma_{j,t} \ge 1$. We argue that
another stationary policy $\bar \mu$, which agrees with the
decisions made by the policy $\mu$ except that $\bar \mu$
charges vehicle $j$ instead of $i$ at the system state $\bx_t$, must
also be optimal.

For the state $\bx_t$,  the optimal stationary policy $\mu$ attains
the minimum in \eqref{eq:Bell} \cite{B11}. To argue that the
stationary policy $\bar \mu$ is optimal, we only need to show
that the stationary policy $\bar \mu$ also attains the minimum
in \eqref{eq:Bell}, i.e.,
\begin{equation}\label{eq:property1}
  \begin{array}{l}
\displaystyle \;\;\; \;g(\v{x}_t,\bar \mu(\v{x}_t)) +
\sum\nolimits_{\bar {\v{x}}_{t+1} \in {\mathcal{X}}}
p_{\v{x}_t,\bar {\v{x}}_{t+1}}
(\bar \mu(\v{x}_t)) h(\bar {\v{x}}_{t+1})\\[6pt]
\displaystyle  \le g(\v{x}_t,\mu(\v{x}_t)) +
\sum\nolimits_{\v{x}_{t+1} \in {\mathcal{X}}}
p_{\v{x}_t,\v{x}_{t+1}} (\mu(\v{x}_t)) h(\v{x}_{t+1}).
\end{array}
\end{equation}

\old{
Note that at state $\v{x}_t$, both policies $\mu$ and $\bar
\mu$ result in the same aggregate action ${A}_t$. It follows
that $g(\v{x}_t,\bar \mu(\v{x}_t)) = g(\v{x}_t,\mu(\v{x}_t))$, since the charging cost depends only
 on $s_t$ and $A_t$, and both vehicles $i$ and $j$

as a result, Eq. \eqref{eq:property1} is equivalent to
\begin{equation}\label{eq:property11}
  \begin{array}{l}
\displaystyle \;\;\; \sum\nolimits_{\bar {\v{x}}_{t+1} \in
{\mathcal{X}}} p_{\v{x}_t,\bar {\v{x}}_{t+1}} (\bar
\mu(\v{x}_t)) h(\bar {\v{x}}_{t+1}) \le \sum\nolimits_{\v{x}_{t+1} \in {\mathcal{X}}}
p_{\v{x}_t,\v{x}_{t+1}} (\mu(\v{x}_t)) h(\v{x}_{t+1}).
\end{array}
\end{equation}}

For every ${\v{x}}_{t+1}$ such that $p_{\v{x}_t,\v{x}_{t+1}}
(\mu(\v{x}_t))>0$, there exists a corresponding system state
$\bar {\v{x}}_{t+1}$ that occurs with equal probability under
the policy $\bar \mu$. The system state $\bar
{\v{x}}_{t+1}$ is the same as $\v{x}_{t+1}$ except that $\bar
\gamma_{i,t+1}= \gamma_{i,t+1}+1$ and  $\bar \gamma_{j,t+1}=
\gamma_{j,t+1}-1$. For every system state ${\v{x}}_{t+1}$ that
occurs with positive probability under the policy $\mu$, we will
show that
\begin{equation}\label{eq:property11}
g(\v{x}_t,\bar \mu(\v{x}_t)) + h(\bar {\v{x}}_{t+1}) \le g(\v{x}_t,\mu(\v{x}_t))+h({\v{x}}_{t+1}),
\end{equation}
which implies the result in \eqref{eq:property1}.

For every system state ${\v{x}}_{t+1}$ that occurs with positive
probability under the policy $\mu$,
 since the state space is finite, there exists a
  policy $(\nu_{t+1},\nu_{t+2},\ldots)$ that attains the
minimum on the right hand side of \eqref{eq:h} for the state
$\bx_{t+1}$. For the policy
$\pi=(\mu,\ldots,\mu,\nu_{t+1},\nu_{t+2},\ldots)$ (a non-stationary
policy that agrees with the stationary policy $\mu$ from stage $0$
through $t$), consider its interchanging policy $\bar
\pi=(\mu,\ldots,\mu, \bar \mu, \bar \nu_{t+1},\bar
\nu_{t+2},\ldots)$  with respect to vehicles $i$ and $j$ of state
$\bx_t$ (cf. Definition \ref{def:update}). In the proof of Theorem
\ref{prop:update}, we have shown that compared to the original
policy $\pi$, the interchanging policy $\bar \pi$ cannot increase the
total cost realized on any trajectory (from stage $t$ through stage
$t+W$)\footnote{We let $W=\max\{\lambda_{i,t},\lambda_{j,t}\}-1$.} that would occur with positive probability (cf. Eq.
\eqref{eq:update2}). It follows that
\begin{equation}\label{eq:property2}
  \begin{array}{l}
\displaystyle\;\;\;
 g(\v{x}_t,\bar \mu(\v{x}_t)) +   g(\bar {\v{x}}_{t+1},\bar \nu_{t+1}(\v{x}_{t+1})) +\E_{\bar \pi}\left\{
{\sum\nolimits_{k = t+2}^{t+W} { g(\bar {\v{x}}_k,  \bar
\nu_k(\bar {\v{x}}_k))}  } \right\} \\[10pt]
\displaystyle \le g(\v{x}_t, \mu(\v{x}_t)) +g(\v{x}_{t+1},
\nu_{t+1}(\v{x}_{t+1})) +\E_{\pi}\left\{
{\sum\nolimits_{k = t+2}^{t+W} { g( {\v{x}}_k, \nu_k( {\v{x}}_k))} }
\right\},
\end{array}
\end{equation}
where the
expectations are over future system states from stage $t+2$ through
$t+W$ induced by the policy $\bar \pi$ and $\pi$, respectively.
 For stages
$k=t+1,\ldots,t+W$, since at least one vehicle ($i$ or $j$) is at a
charging facility, the special system state\footnote{The
special system state is introduced  in Eq. \eqref{eq:spe}, and the cost
function $\tilde g(\v{x}_k, \v{a}_k)$ is defined prior to Eq.
\eqref{eq:h}.} is not reached, and therefore $\tilde g(\v{x}_k,
\v{a}_k)=g(\v{x}_k, \v{a}_k)$ for any action $\v{a}_k$.

\old{
Note that $g(\v{x}_t,\bar \mu(\v{x}_t)) =
g(\v{x}_t,\mu(\v{x}_t))$. The inequality in \eqref{eq:property2}  is
equivalent to
\begin{equation}\label{eq:property3}
  \begin{array}{l}
\displaystyle\;\;\;\;  \tilde g(\bar {\v{x}}_{t+1},\bar
\nu_{t+1}(\v{x}_{t+1})) +\E_{\bar \pi}\left\{ {\sum\limits_{k =
t+2}^{t+W}( {\tilde g(\bar {\v{x}}_k, \bar \nu_k(\bar
{\v{x}}_k))} -
\lambda) } \right\} \\[12pt]
\displaystyle \le \tilde g(\v{x}_{t+1}, \nu_{t+1}(\v{x}_{t+1}))
+\E_{\pi}\left\{ {\sum\limits_{k = t+2}^{t+W}( {\tilde g( {\v{x}}_k,
\nu_k( {\v{x}}_k))} - \lambda) } \right\}.
\end{array}
\end{equation}}

 After stage $t+W$, the two policies $\pi$ and $\bar \pi$
  result in the same expected cost. We therefore have
\begin{equation}\label{eq:property4}
  \begin{array}{l}
\displaystyle\;\;\;\;  g(\v{x}_t,\bar \mu(\v{x}_t)) +h
(\bar {\v{x}}_{t+1}) + \lambda\\[4pt]
 \displaystyle
\le g(\v{x}_t,\bar \mu(\v{x}_t))+\tilde g(\bar {\v{x}}_{t+1},\bar
\nu_{t+1}(\v{x}_{t+1}))   +\mathop {\limsup }\limits_{T \to \infty }
\E_{\bar \pi}\left\{ {\sum\nolimits_{k = t+2}^{t+T}( {\tilde
g(\bar {\v{x}}_k,
\bar \nu_k(\bar {\v{x}}_k))} - \lambda) } \right\}\\[9pt]
 \displaystyle
\le g(\v{x}_t,\mu(\v{x}_t)) + \tilde g({\v{x}}_{t+1},
\nu_{t+1}(\v{x}_{t+1}))   + \mathop {\limsup }\limits_{T \to \infty }
\E_{\pi}\left\{ {\sum\nolimits_{k = t+2}^{t+T}( {\tilde g(
{\v{x}}_k,
 \nu_k(\bar {\v{x}}_k))} -
\lambda) } \right\}\\[9pt]
 \displaystyle = g(\v{x}_t,\mu(\v{x}_t))+h
({\v{x}}_{t+1}) + \lambda,
 \end{array}
\end{equation}
where the first inequality follows from the definition of
$h(\bar {\v{x}}_{t+1})$ in \eqref{eq:h}, the second inequality
 follows from
  \eqref{eq:property2}, and the last equality is true because the policy
$\pi=(\mu,\ldots,\mu,\nu_{t+1},\nu_{t+2},\ldots)$ attains the
minimum on the right hand side of \eqref{eq:h} with an initial state
$\bx_{t+1}$.

We have shown the inequality in \eqref{eq:property11} holds for
 for every ${\v{x}}_{t+1}$ such that
$p_{\v{x}_t,\v{x}_{t+1}} (\mu(\v{x}_t))>0$. It follows from
\eqref{eq:property1} that
 the stationary policy $\bar \mu$ attains
the minimum in \eqref{eq:Bell}, and is therefore optimal. Since the
state space is finite, by repeating this interchange argument for finitely many times, we can
construct an optimal stationary policy that follows the LLLP
principle at all system states.

\bibliographystyle{alpha}

\end{document}